\documentclass[11pt,leqno]{amsart}
\usepackage{amscd}
\usepackage{color}
\usepackage[symbol]{footmisc}
\usepackage{amssymb}
\usepackage{amsfonts}
\usepackage{latexsym}
\usepackage{verbatim}

\newcommand{\R}{\mathbb{R}}

\renewcommand{\epsilon}{\varepsilon}
\renewcommand{\theta}{\vartheta}
\renewcommand{\phi}{\varphi}
\renewcommand{\Re}{\mathrm{Re}}

\theoremstyle{plain}
\newtheorem{teor}{Theorem}[section]
\newtheorem{prop}[teor]{Proposition}
\newtheorem{lem}[teor]{Lemma}

\newtheorem{rem}[teor]{Remark}

\theoremstyle{definition}

\title[Second order estimates for the quaternionic Calabi-Yau problem]{A remark on the second order estimates for the quaternionic Calabi-Yau problem on hyperk\"ahler manifolds}
\makeatletter
\@namedef{subjclassname@2020}{%
 \textup{2020} Mathematics Subject Classification}
\makeatother
\begin{document}
\thanks{This work was supported by GNSAGA of INdAM and by the PRIN project \lq\lq Differential-geometric aspects of manifolds via Global Analysis''  20225J97H5}

\address{Giovanni Gentili\\
Dipartimento di Matematica ``G. Peano'', Universit\`{a} degli studi di Torino \\
Via Carlo Alberto 10, Torino (Italy)}
\email{giovanni.gentili@unito.it}

\medskip

\address{Luigi Vezzoni \\
Dipartimento di Matematica ``G. Peano'', Universit\`{a} degli studi di Torino \\
Via Carlo Alberto 10, Torino (Italy)}

\email{luigi.vezzoni@unito}

\author{Giovanni Gentili and Luigi Vezzoni}

\date{\today}

\begin{abstract}
We revisit the second order estimate for solutions to the quaternionic Calabi--Yau problem on hyperkähler manifolds, originally established by Dinew and Sroka in~\cite{Dinew-Sroka}. In this note, we present a simplified argument to derive the estimate.
\end{abstract}

\maketitle

\section{Introduction}
The present paper focuses on the Calabi-Yau problem in hypercomplex geometry. The problem was proposed by Alesker and Verbitsky \cite{Alesker-Verbitsky (2010)} as a natural adaptation of the classical Calabi-Yau problem to hypercomplex geometry and it has been the research topic of several papers (see e.g. \cite{Alesker (2013),Alesker-Shelukhin (2017),Alesker-Verbitsky (2010),BGV,Ibero,Dinew-Sroka,GV,GV2,Sroka,Sroka22,SThesis} and the references therein). 

Given a compact hyperhermitian manifold $(M,I,J,K,\Omega)$ and a smooth function $F$ on $M$, the {\em hypercomplex Calabi-Yau problem} consists in finding $(\varphi,b)\in C^{\infty}(M)\times \R$ such that  
\begin{equation}\label{CY}
(\Omega+\partial\partial_J\phi)^n={\rm e}^{F+b}\Omega^n\,,\quad \Omega+\partial\partial_J\phi>0\,,\quad {\rm sup}_M \varphi=0\,. 
\end{equation}

Geometric applications of the problem are mostly relevant in the HKT case, i.e. when $\partial \Omega=0$ \cite{Grantcharov-Poon (2000),Howe-Papadopoulos (2000)}, since in this case the problem is strictly related to the seek of canonical hyperhermitian metrics. Indeed, in the HKT case under the assumption that the canonical bundle of $(M,I)$ is holomorphically trivial, the solvability of the problem would imply that in the quaternionic Bott-Chern class $[\Omega]_{\partial \partial_J}$ there is always a unique {\em balanced HKT} metric \cite{Verbitsky (2009)}. Such a metric is necessarily {\em Chern-Ricci flat} \cite{Ibero}.

The problem has at most one solution $(\varphi,b)$ and \eqref{CY} is usually called {\em quaternionic Monge-Amp\`ere equation}. The solvability of the problem in its full generality is still an open question, but some important advances have been established in the literature. 

According to the strategy of Yau to prove the Calabi conjecture \cite{Yau}, the natural approach for solving the problem is via the continuity method. 
The $C^0$ estimate is proved \cite{Alesker-Verbitsky (2010),Alesker-Shelukhin (2017),Sroka,Sroka22}, while from \cite[Theorem 3.1]{BGV} it follows that, in order to solve the problem, it is enough to show an upper bound on the Chern Laplacian of the solution. 

This last estimate was proved by Alesker on flat hyperk\"ahler manifolds \cite{Alesker (2013)} and more recently by Dinew and Sroka on hyperk\"ahler manifolds \cite{Dinew-Sroka}. 

The purpose of the present note is to present a simplified approach for proving the second order estimate of Dinew and Sroka. 
Our key observation, Lemma \ref{fund} below, is that if $\hat g$ is an hyperk\"ahler metric, then 
$$
{g}^{r\bar s}\hat R_{r\bar si\bar j}=0
$$
for {\em every hyperhermitian metric} $g$ on $M$, where $\hat R$ is the curvature of $\hat g$. This allows us to adopt a simple approach to prove the estimate. 
We already used the same observation in the parabolic case in \cite{Ibero}. 

More in detail, with respect to \cite{Dinew-Sroka}, we use a different test function $Q$, defined in \eqref{Q} below, in the maximum principle argument. Our choice yields fewer terms in the computations thus allowing us to avoid the usage of the critical equation ($d Q=0$ at a maximum point). Finally, in our argument we reinterpret the vanishing of certain terms by inspecting explicitly some properties of the curvature of the hyperk\"ahler metric, whereas in \cite{Dinew-Sroka} this was achieved somewhat implicitly. See Remark \ref{lastRem} for more details on this.

\bigskip 
\noindent {\bf Notation.} In this paper, we adopt the following notation: all tensor components are expressed with respect to complex coordinates that are holomorphic with respect to the first complex structure~$I$. We use a comma in the indices to denote partial derivatives. When differentiating functions, the comma is omitted.

\bigskip 
\noindent {\bf Acknowledgements.}
The subject of the present note was presented by the second-named author at the conference ``The 8\textsuperscript{th} Workshop on Complex Geometry and Lie Groups'' held in Osaka from March 10\textsuperscript{th} to 14\textsuperscript{th}, 2025. The authors are extremely grateful to the organizers — Anna Fino, Ryushi Goto, Keizo Hasegawa, Hisashi Kasuya, and Alberto Raffero — for hosting such a beautiful and well-organized conference.

\medskip 
\noindent Moreover the authors are very grateful to S\l{}awomir Dinew and Marcin Sroka for very useful discussions.

\section{Preliminaries}

A hyperhermitian manifold is a smooth manifold of real dimension $4n$ equipped with a triple of complex structures $I,J,K$ satisfying the quaternionic relations
\begin{equation}\label{q}
IJ=-JI=K\,. 
\end{equation}
Examples of hyperhermitian manifolds are given by: the {\em quaternionic vector space} $\mathbb H^n$, tori $\mathbb H^n/\Gamma$, where $\Gamma$ is a lattice in $\mathbb H^n$; Hopf manifolds of the form $((\mathbb C^n\times \mathbb C^n)\backslash \{(0,0)\})/\langle r\rangle $, where $r$ is a real number with $0< r<1$, and $\langle r\rangle$ is the group generated by the action 
$$
(z_\alpha,w_\alpha)\mapsto (r{\rm e}^{i\theta_\alpha}z_{\alpha},r{\rm e}^{-i\theta_\alpha}w_\alpha)
$$
and $\theta_1,\dots,\theta_n\in \mathbb{R}$.

\medskip 
A fundamental tool in  hypercomplex geometry is the so-called {\em Obata connection} introduced in \cite{Obata (1956)} as the unique torsion-free affine connection $\nabla$  which preserves all the three complex structures. This connection can be used to detect important properties of the hypercomplex structure. For instance it is flat if and only if the manifold  is locally isomorphic to $\mathbb H^n$. That implies strong restrictions on the geometry of the manifold since it turns out to be affine \cite{Sommese}. Moreover the condition ${\rm Hol}(\nabla)\subseteq {\rm SL}(n,\mathbb H)$ implies that each complex structure of the hyperhermitian structure has holomorphically trivial canonical bundle. 

\medskip 
A hyperhermitian metric on a hypercomplex manifold $(M,I,J,K)$ is a Riemannian metric $g$ which is compatible with each complex structure, i.e. 
$$
g(I\cdot,I\cdot)=g(J\cdot,J\cdot)=g(K\cdot,K\cdot)\,. 
$$
The metric $g$ induces the three fundamental forms 
$$
\omega(\cdot,\cdot):=g(I\cdot ,\cdot)\,,\quad \omega_J(\cdot,\cdot):=g(J\cdot ,\cdot)\,,\quad \omega_K(\cdot,\cdot):=g(K\cdot ,\cdot)\,.
$$
The quaternionic relations \eqref{q} imply that  the $2$-form 
$$
\Omega:=\frac12(\omega_J+i\omega_K)
$$
is of type $(2,0)$ with respect to $I$ and it is easy to see that satisfies the {\em q-real} condition 
$$
\Omega(J\cdot ,J\cdot)=\overline{\Omega(\cdot,\cdot)\,.}
$$
Furthermore $\Omega$ is \mbox{positive}, in the sense that 
$$
\Omega(Z,J\bar Z)>0
$$
for every non-zero vector field $Z$ of type $(1,0)$ with respect to $I$. Vice versa, every $\Omega\in \Lambda^{2,0}_I$  which is $q$-real and positive induces a hyperhermitian metric $g$ via the formula 
\[
g=2\Re (\Omega(\cdot,J\cdot))\,.
\]
On the other hand the form $\omega$ is a positive real form of type $(1,1)$ with respect to $I$ and it satisfies the $J$-anti-invariant relation 
$$
\omega(J\cdot,J\cdot)=-\omega(\cdot,\cdot) \,.
$$
Vice versa, if $\omega\in \Lambda^{1,1}_{I}$ is real positive and $J$-anti-invariant the induced Hermitian metric is hyperhermitian with respect to $(I,J,K)$.   Hence the following three spaces are in bijective correspondence  
\begin{itemize}
\item hyperhermitian metrics; 

\vspace{0.1cm}
\item $q$-real positive forms in $\Lambda^{2,0}_I$;

\vspace{0.1cm}
\item real positive $J$-anti-invariant forms in $\Lambda^{1,1}_I$.
\end{itemize}

A hyperhermitian metric is hyperk\"ahler if and only if the induced $\Omega $ is closed. In this case the three Bismut connections induced by $(g,I)$, $(g,J)$ and $(g,K)$ agree with the Levi-Civita connection of $g$. In \cite{Howe-Papadopoulos (2000)} Howe and Papadopoulos generalized the notion of hyperk\"ahler manifold to geometry with torsion by introducing {\em HKT manifolds}. An HKT manifold can be defined as a hyperhermitian manifold $(M,I,J,K,g)$ 
whose quaternionic form $\Omega$ satisfies 
$$
\partial \Omega=0\,,
$$
where $\partial$ is with respect to $I$ (this is in fact a characterization given in \cite{Grantcharov-Poon (2000)}).

In hypercomplex geometry the role that the $\partial \bar \partial$ operator plays in complex geometry is taken by the operator $\partial \partial_J$, where $\partial_J:= J^{-1}\bar \partial J\colon \Lambda^{r,0}_I\to \Lambda^{r+1,0}_I$. 
A smooth function $\varphi$ on $(M,I,J,K,\Omega)$ is {\em $\Omega$-plurisubharmonic} if $\Omega_\phi:=\Omega+\partial \partial_J\phi$ is $q$-positive. In such a case $\Omega_\phi$ gives a hyperhermitian metric on $(M,I,J,K)$. The $I$-fundamental form $\omega_\varphi$ of the metric induced by $\Omega_\phi$ takes the following expression 
\begin{equation}\label{fform}
\omega_\phi:=\omega+\frac{1}{2}\left(i\partial \bar\partial \varphi-iJ\partial \bar\partial \varphi\right)
\end{equation}
and, as observed in \cite[Section 2.3]{Dinew-Sroka}, problem \eqref{CY} is equivalent to the following equation stated in terms of $\omega_\phi$:
\begin{equation}\label{eqrewrite}
\omega_{\varphi}^{2n}={\rm e}^{2F+2b}\omega^{2n}\,,\quad \omega_{\varphi}>0\,,\quad {\rm sup}_M \varphi=0\,.
\end{equation}
Note that if $\Omega$ is HKT, then $\Omega_{\varphi}$ is HKT for every positive $\varphi$ and the class 
$$
[\Omega]_{\partial \partial_J}:=\{\Omega+\partial\partial_J f\,\,:\,\, f\in C^{\infty}(M)\}
$$
plays the role that the Bott-Chern cohomology class of $\omega$ plays in K\"ahler geometry. 

\section{Hyperhermitian metrics on hyperk\"ahler manifolds}
Let $(M,I,J,K)$ be a hyperhermitian manifold and denote by $R$ the curvature of the Obata connection $ \nabla$. Since $\nabla$ is torsion-free and preserves $I$ we can find $I$-holomorphic coordinates $\{z^r\}$ around $x_0$ that are geodesic with respect to $\nabla$ at $x_0$. Let
\[
R(X,Y)=\nabla_X\nabla_Y-\nabla_Y\nabla_X-\nabla_{[X,Y]}
\]
be the curvature endomorphisms, then we denote the components of the curvature as
\[
R\left(\frac{\partial}{\partial \bar z^j}, \frac{\partial}{\partial z^k}\right) \frac{\partial}{\partial \bar z^l}=R_{\bar jk\bar l}{}^{\bar i} \frac{\partial}{\partial \bar z^i}\,.
\]

The next useful proposition follows from \cite[Remark 2.13]{Dinew-Sroka}: 

\begin{prop}\label{pre}
With respect to $\{z^r\}$ we have
\begin{equation}\label{pre1}
J_{r,k}^{\bar s}=J^{\bar s}_{k,r}\,.
\end{equation}
Furthermore
\begin{eqnarray}
&&\label{pre2}  J_{r, k}^{\bar s}=J_{r,\bar k}^{\bar s}=0\,,\\
&&\label{pre3}  J_{	\bar j,k\bar l}^s J_{s}^{\bar i}=-
J_{	\bar j}^s J_{s,k\bar l}^{\bar i}
\,,\\
&&\label{pre4}   R_{\bar jk\bar l}\,^{\bar i}=J_{ a}^{\bar i}J_{   \bar l,k\bar j}^{a}
\end{eqnarray}
at $x_0$.   
\end{prop}
\begin{proof}
Since $\nabla$ is torsion-free and preserves $I$ the mixed covariant derivatives $\nabla_{\bar r}\partial_{z^{s}}$ vanish and we have 
$$
 0=(\nabla_{k}J)\partial_{z^r}=\left (J_s^{\bar m}\Gamma_{kr}^s -
 J_{r,k}^{\bar m}\right)\partial_{\bar z^m}
$$ 
which implies 
$$
 J_{r,k}^{\bar m}=J_s^{\bar m}\Gamma_{kr}^s\quad \mbox{ and } \quad 
 \Gamma_{kl}^i=-J_{\bar m}^{i}J_{l,k}^{\bar m}\,. 
$$
Since $\Gamma_{kl}^i=\Gamma_{lk}^i$, then \eqref{pre1} follows.

Since the Christoffel symbols of $\nabla$ vanish at $x_0$, the first derivatives of $J_{r}^{\bar{s}}$ vanish at $x_0$ and \eqref{pre2} follows. 

To prove \eqref{pre3} we just compute 
$$
0=(J_{	\bar j}^s J_{s}^{\bar i})_{k\bar l}=J_{	\bar j,k\bar l}^s J_{s}^{\bar i}+
J_{	\bar j}^s J_{s,k\bar l}^{\bar i}
$$
at $x_0$.

For the last formula we observe 
$$
R_{\bar jk\bar l}\,^{\bar i}=-\Gamma_{\bar j \bar l,k}^{\bar i}=\partial_{k}(J_{a}^{\bar i}J_{\bar l,\bar j}^a) =J_{ a}^{\bar i}J_{   \bar l,k\bar j}^{a}
$$
at $x_0$, and the claim follows. 
\end{proof}

\begin{lem}\label{fund}
Let $g$ be a hyperhermitian metric. Then
\begin{equation}\label{fund1}
g^{i\bar j}R_{\bar j i \bar q}\,^{\bar a}=0
\end{equation}
at $x_0$.
In particular if  $\hat g$  is hyperk\"ahler, then
$$
g^{i\bar j}\hat g_{i\bar k}R_{\bar sr\bar j}\,^{\bar k}=0\,,
$$
for {\em every} hyperhermitian metric $g$ on $(M,I,J,K)$.
\end{lem}
\begin{proof}
We compute using Proposition \ref{pre}
\[
g^{i\bar j}R_{\bar j i \bar q}\,^{\bar a}=g^{i\bar j}J_b^{\bar a}J^b_{\bar q,i\bar j}=g^{i\bar j}J_b^{\bar a}J^b_{\bar j,i\bar q}=-g^{i\bar j}J_{b,i\bar q}^{\bar a}J^b_{\bar j}\,,
\]
where we used, in order, \eqref{pre4}, \eqref{pre1}, \eqref{pre3}. Now, since $g$ is hyperhermitian we have $g^{i\bar j}=g^{r\bar s}J_r^{\bar j}J_{\bar s}^i$ and thus
\[
g^{i\bar j}R_{\bar j i \bar q}\,^{\bar a}=-g^{r\bar s}J_r^{\bar j}J_{\bar s}^iJ_{b,i\bar q}^{\bar a}J^b_{\bar j}\\
=g^{b\bar s}J_{\bar s}^iJ_{b,i\bar q}^{\bar a}=g^{b\bar s}J_{\bar s}^iJ_{i,b\bar q}^{\bar a}=-g^{b\bar s}J_{\bar s,b\bar q}^iJ_{i}^{\bar a}=-g^{b\bar s}J_{\bar q,b\bar s}^iJ_{i}^{\bar a}=-g^{b\bar s}R_{\bar s b \bar q}\,^{\bar a}
\]
where, starting from the third equality, we used \eqref{pre1}, \eqref{pre3}, \eqref{pre1}, and finally \eqref{pre4} again. The identity \eqref{fund1} is proved. The last sentence follows since if $\hat g$ is hyperk\"ahler, then $R$ is the curvature of $\hat g$ and
\[
g^{i\bar j}\hat g_{i\bar k}R_{\bar sr\bar j}\,^{\bar k}=g^{i\bar j}\hat g_{r \bar k}R_{\bar ji\bar s}\,^{k}\,. \qedhere
\]
\end{proof}

\begin{rem}{\em 
From the proof of Lemma \ref{fund} it is clear that \eqref{fund1} still holds if instead of tracing with respect to a hyperhermitian metric we trace with respect to a hyperhermitian symmetric tensor. Also, an alternative perspective of Lemma \ref{fund} is as follows. Let $g$ be a hyperhermitian metric and $\{Z_1,\dots, Z_{2n}\}$ an orthonormal frame of type $(1,0)$ with respect to $I$ that is $J$-adapted, namely $Z_{2i}=J\bar Z_{2i-1}$. Let
\[
R(X,Y)=\nabla_X\nabla_Y-\nabla_Y\nabla_X-\nabla_{[X,Y]}
\]
be the curvature endomorphisms. Then the trace in \eqref{fund1} is essentially
\[
\sum_{j=1}^{2n} R(Z_j,\bar Z_j)=\frac{1}{2}\sum_{j=1}^{2n} \left( R(Z_j,\bar Z_j) + R(J \bar Z_j,J Z_j) \right)=\frac{1}{2}\sum_{j=1}^{2n} \left( R(Z_j,\bar Z_j) - R(J Z_j,J \bar Z_j) \right)=0\,,
\]
where the last equality is due to the identity $R(JX,JY)=R(X,Y)$.}
\end{rem}

The following lemma will be useful

\begin{lem}\label{fund2}
Let $\hat g$ be a hyperk\"ahler metric and let $g$ be a hyperhermitian metric. Then the following formulas hold 
\begin{eqnarray}
&& \label{1} g^{i\bar j}\hat g^{r\bar s}J_{r,i\bar j}^{\bar a} J_{\bar s}^b=0\,,\\
&& \label{2} g^{i\bar j}\hat g^{r\bar s}J_{r}^{\bar a}J_{\bar s,i\bar j}^b=0\,,\\
&& \label{3} g^{i\bar j}\hat g^{r   \bar s }J_{i}^{\bar b}J_{\bar j,r\bar s}^a=0\,,\\
&& \label{4} g^{i\bar j}\hat g^{r   \bar s }J_{i,r\bar s}^{\bar b}J_{\bar j}^a=0
\end{eqnarray}
at $x_0$. 
\end{lem}
\begin{proof}
Using, hyperhermitianity of $\hat g$ and $J^2=-\mathrm{Id}$ we have 
$$
g^{i\bar j}\hat g^{r\bar s}J_{r,i\bar j}^{\bar a} J_{\bar s}^b=g^{i\bar j}\hat g^{p\bar q}J_p^{\bar s} J_{\bar q}^rJ_{r,i\bar j}^{\bar a} J_{\bar s}^b=
-g^{i\bar j}\hat g^{b\bar q} J_{\bar q}^rJ_{r,i\bar j}^{\bar a}=g^{i\bar j}\hat g^{b\bar q}R_{i\bar j\bar q}\,\!^{\bar a}=0
$$
at $x_0$, where the second-to-last equality is due to \eqref{pre3} and \eqref{pre4}, while the last step follows from Lemma \ref{fund}. 
Similarly, but using hyperhermitianity of $g$, we deduce
$$
g^{i\bar j}\hat g^{r   \bar s }J_{i}^{\bar b}J_{\bar j,r\bar s}^a=
g^{p\bar q}J_{p}^{\bar j}J_{\bar q}^i\hat g^{r   \bar s }J_{i}^{\bar b}J_{\bar j,r\bar s}^a=-
g^{p\bar b}\hat g^{r   \bar s }J_{p}^{\bar j}J_{\bar j,r\bar s}^a=
g^{p\bar b}\hat g^{r   \bar s }R_{\bar s rp}\,^a=0
$$
at $x_0$. Hence \eqref{1} and \eqref{3} follow. The other two relations can be obtained by conjugation.
\end{proof}

\section{A proof of the second order estimate}

\begin{teor}
Let $\varphi$ be a solution to \eqref{CY}, then $\hat \Delta\varphi:=\hat g^{r\bar s}\varphi_{r\bar s}$ has an a  priori upper 
bound.
\end{teor}

\begin{proof}
Let 
\begin{equation} \label{Q}
Q:=\mathrm{tr}_{\hat g}g_\phi -A\phi\,,
\end{equation}
where $A$ is a constant to be chosen later.

Denote $\Delta_\phi f:=g_\phi^{r\bar s}f_{r\bar s}$ the Chern-Laplacian with respect to $g_\phi$. We have
$$
\Delta_\varphi Q=\Delta_\varphi\mathrm{tr}_{\hat g}g_\phi-A\Delta_\varphi\phi=
\Delta_\varphi\mathrm{tr}_{\hat g}g_\phi-2nA+A \mathrm{tr}_{g_\varphi}g\,.
$$

Assume that $Q$ achieves a maximum at $x_0\in M$ and consider $I$-holomorphic coordinates around $x_0$ such that 
$$
\hat g_{r\bar s,k}=J_{r,k}^{\bar s}=J_{r,\bar k}^{\bar s}=0
$$
at $x_0$. To find such coordinates simply take normal coordinates for the hyperk\"ahler metric $\hat g$ at $x_0$. Indeed, since the Levi-Civita connection of $\hat g$ coincides with the Obata connection it is clear from Proposition \ref{pre} that \eqref{pre2} holds at $x_0$. Then we directly compute at $x_0$   
$$
\begin{aligned}
\Delta_\phi \mathrm{tr}_{\hat g}g_\phi&= g_\phi^{i\bar j} \left(((\hat g^{r\bar s})_{,i\bar j}g_{\phi \,r\bar s}+\hat g^{r\bar s} g_{\phi\, r\bar s, i \bar j}\right)\\
&=-g_\phi^{i\bar j}\hat g^{a\bar s}\hat g^{r\bar b}\hat g_{a\bar b,i\bar j}g_{\varphi\,r\bar s}+\frac{1}{2}g_\phi^{i\bar j}\hat g^{r\bar s}\left(2g_{r\bar s,i\bar j}+\phi_{r\bar s i\bar j}+J^{\bar a}_rJ^b_{\bar s}\phi_{b\bar a i\bar j} +J_{r,i\bar j}^{\bar a}J_{\bar s}^b\phi_{b\bar a}+J_{r}^{\bar a}J_{\bar s,i\bar j}^b\phi_{b\bar a}\right)\\
&=g_\phi^{i\bar j}\hat g^{a\bar s}\hat g^{r\bar b}\hat R_{a\bar bi\bar j}g_{\varphi\,r\bar s}
+g_\phi^{i\bar j}\hat g^{r\bar s}\left(g_{r\bar s,i\bar j}+\phi_{r\bar s i\bar j}\right) +\frac{1}{2}g_\phi^{i\bar j}\hat g^{r\bar s}\left(J_{r,i\bar j}^{\bar a}J_{\bar s}^b+J_{r}^{\bar a}J_{\bar s,i\bar j}^b\right) \phi_{b\bar a}\,,
\end{aligned}
$$
where $\hat R_{a\bar b i \bar j}=\hat g_{k\bar j} R_{a\bar b i}\,^{k}$ and we used the fact that $\hat g$ is $J$-Hermitian, so that $\hat g^{r\bar s}J_r^{\bar a}J_{\bar s}^b=\hat g^{b\bar a}$. Hence Lemmas \ref{fund} and \ref{fund2} imply 
\begin{equation}\label{Delta}
\begin{aligned}
\Delta_\phi \mathrm{tr}_{\hat g}g_\phi=
g_\phi^{i\bar j}\hat g^{r\bar s}\left(g_{r\bar s,i\bar j}+\phi_{r\bar s i\bar j}\right) \,.
\end{aligned}
\end{equation} 
Next we use equation \eqref{CY} to cancel the fourth order derivatives of $\phi$ in \eqref{Delta}. Using \eqref{fform} and \eqref{eqrewrite} we have 
$$
F+b=\log\frac{(\Omega+\partial\partial_J\varphi)^n}{\Omega^n}=\frac12\log\frac{(\omega+\tfrac12 \partial\bar \partial\varphi-\tfrac12 iJ\partial\bar \partial\varphi )^{2n}}{\omega^{2n}}
$$
and 
$$
\begin{aligned}
\partial_{r} F&=\frac{1}{2}\left( g_{\varphi}^{i\bar j}g_{\phi\, i\bar j,r}-g^{i\bar j}g_{i\bar j,r}\right)\\
&=\frac{1}{2}g_{\varphi}^{i\bar j}\left(g_{i\bar j,r}+\tfrac12 \varphi_{ri\bar j}
+\tfrac12 J_{i}^{\bar b}J_{\bar j}^a\varphi_{ra\bar b}
+\tfrac12 J_{i}^{\bar b}J_{\bar j,r}^a\varphi_{ra\bar b}
+\tfrac12 J_{i,r}^{\bar b}J_{\bar j}^a\varphi_{a\bar b}\right)-\frac{1}{2}g^{i\bar j}g_{i\bar j,r}
\end{aligned}
$$
as well as
\begin{multline*}
\partial_{r}\partial_{\bar s} F=
-\frac{1}{2}g_{\varphi}^{i\bar b}g_{\varphi}^{a\bar j}g_{\varphi\,a\bar b,\bar s}g_{\varphi\,i\bar j,r}\\ 
+\frac{1}{2}g_{\varphi}^{i\bar j}(g_{i\bar j,r\bar s}
+\tfrac12 \varphi_{r\bar s i\bar j }+\tfrac12 J_{i}^{\bar b}J_{\bar j}^a\varphi_{r\bar s a\bar b}
+\tfrac12 J_{i}^{\bar b}J_{\bar j,r\bar s}^a\varphi_{ra\bar b}
+\tfrac12 J_{i,r\bar s}^{\bar b}J_{\bar j}^a\varphi_{a\bar b})-H_{r\bar s}\,,
\end{multline*}
where $H_{r\bar s}$ only depends on $g$. It follows 
$$
\begin{aligned}
\partial_{r}\partial_{\bar s} F=&-\frac{1}{2}g_{\varphi}^{i\bar b}g_{\varphi}^{a\bar j}g_{\varphi\,a\bar b,\bar s}g_{\varphi\,i\bar j,r}
+\frac{1}{2}g_{\varphi}^{i\bar j}(g_{i\bar j,r\bar s}
+ \varphi_{r\bar s i\bar j }
+\tfrac12 J_{i}^{\bar b}J_{\bar j,r\bar s}^a\varphi_{ra\bar b}
+\tfrac12 J_{i,r\bar s}^{\bar b}J_{\bar j}^a\varphi_{a\bar b})-H_{r\bar s}
\end{aligned}
$$
and 
$$
\begin{aligned}
\hat \Delta F=&
-\frac{1}{2}\hat g^{r   \bar s }g_{\varphi}^{i\bar b}g_{\varphi}^{a\bar j}g_{\varphi\,a\bar b,\bar s}g_{\varphi\,i\bar j,r}
+\frac{1}{2}\hat g^{r   \bar s }g_{\varphi}^{i\bar j}(g_{i\bar j,r\bar s}
+ \varphi_{r\bar s i\bar j }
+\tfrac12 J_{i}^{\bar b}J_{\bar j,r\bar s}^a\varphi_{ra\bar b}
+\tfrac12 J_{i,r\bar s}^{\bar b}J_{\bar j}^a\varphi_{a\bar b})-\hat g^{r   \bar s }H_{r\bar s}\,.
\end{aligned}
$$
Hence Lemma \ref{fund2} implies 

$$
\begin{aligned}
\hat \Delta F=&-\frac{1}{2}\hat g^{r   \bar s }g_{\varphi}^{i\bar b}g_{\varphi}^{a\bar j}g_{\varphi\,a\bar b,\bar s}g_{\varphi\,i\bar j,r}
+\frac{1}{2}\hat g^{r   \bar s }g_{\varphi}^{i\bar j}(g_{i\bar j,r\bar s}
+ \varphi_{r\bar s i\bar j })-\hat g^{r   \bar s }H_{r\bar s}\,.
\end{aligned}
$$
We can then write 
$$
\hat g^{r   \bar s }g_{\varphi}^{i\bar j}\varphi_{r\bar s i\bar j }=
\hat g^{r   \bar s }g_{\varphi}^{i\bar b}g_{\varphi}^{a\bar j}g_{\varphi\,a\bar b,\bar s}g_{\varphi\,i\bar j,r}
-\hat g^{r   \bar s }g_{\varphi}^{i\bar j}g_{i\bar j,r\bar s}
+h\,,
$$
where $h$ is a smooth function on $M$ which does not depend on $\phi$. Thus \eqref{Delta} becomes 
\begin{equation*}
\begin{aligned}
\Delta_\phi \mathrm{tr}_{\hat g}g_\phi= \hat g^{r\bar s} g_{\varphi}^{i\bar b}g_{\varphi}^{a\bar j}g_{\varphi\,a\bar b,\bar s}g_{\varphi\,i\bar j,r}
+g_\phi^{i\bar j}\hat g^{r\bar s}\left(g_{r\bar s,i\bar j}-g_{i\bar j,r\bar s}\right)+h\,.
\end{aligned}
\end{equation*} 

Condition 
$$
0\geq \Delta_\varphi\mathrm{tr}_{\hat g}g_\phi-2nA+A \mathrm{tr}_{g_\varphi}g\quad  \mbox{ at } x_0
$$
reads as 
$$
0\geq  
\hat g^{r\bar s} g_{\varphi}^{i\bar b}g_{\varphi}^{a\bar j}g_{\varphi\,a\bar b,\bar s}g_{\varphi\,i\bar j,r}
+g_\phi^{i\bar j}\hat g^{r\bar s}\left(g_{r\bar s,i\bar j}-g_{i\bar j,r\bar s}\right)-2nA+A \mathrm{tr}_{g_\varphi}g+h
$$
which implies, after dropping the non-negative quadratic term (the first term in the inequality above),
$$
\begin{aligned}
0\geq &-C{\rm tr}_{g_\phi}\hat g+A \mathrm{tr}_{g_\varphi}g-C-2nA \quad \mbox{ at }x_0\,,
\end{aligned}
$$
where $C$ is a positive constant depending on $g$ and $h$ only.  
Using 
$$
\mathrm{tr}_{g_\varphi}g\geq A'\,\mathrm{tr}_{g_\varphi}\hat g\,,
$$
which holds for a positive constant $A'$ depending on $g$ and $\hat g$ only,  
we have 
$$
0\geq (AA'-C){\rm tr}_{g_\phi}\hat g-C-2nA \quad \mbox{ at }x_0\,.
$$
Hence by taking $A=\tfrac{C}{A'}+1$ we obtain 
$$
{\rm tr}_{g_\phi} \hat g\leq C \quad \mbox{ at }x_0\,.
$$

Now 
$$
{\rm tr}_{\hat g} g_\phi \leq \frac{1}{(2n-1)!} ({\rm tr}_{g_\phi }\hat g)^{2n-1}
\frac{\omega^{2n}_\phi }{\hat \omega^{2n}}
$$
implies 
$$
{\rm tr}_{\hat g} g_\varphi\leq C \quad \mbox{ at }x_0\,.
$$

It follows
$$
\mathrm{tr}_{\hat g}g_\phi -A\phi=Q\leq Q(x_0)=\mathrm{tr}_{\hat g}g_\phi(x_0) -A\phi(x_0)\leq C+A\|\phi\|_{C^0}\,,
$$
and since $\varphi$ satisfies an a priori $C^0$ bound the claim follows.  
\end{proof}

\begin{rem}\label{lastRem}
\emph{Our formula \eqref{Delta} corresponds to equation (4.5) in \cite{Dinew-Sroka}. 
The difference between the two approaches lies in the fact that Dinew and Sroka implicitly prove that the sum
\begin{equation}\label{eqns}
g_\phi^{i\bar j}\hat g^{a\bar s}\hat g^{r\bar b}\hat R_{a\bar bi\bar j}g_{\varphi\,r\bar s}
+\frac{1}{2}g_\phi^{i\bar j}\hat g^{r\bar s}\left(J_{r,i\bar j}^{\bar a}J_{\bar s}^b+J_{r}^{\bar a}J_{\bar s,i\bar j}^b\right) \phi_{b\bar a}
\end{equation}
vanishes, whereas we show the vanishing of each term in \eqref{eqns} separately using the curvature identities proved in Lemmas \ref{fund} and \ref{fund2}.}
\end{rem}

\end{document}